\documentclass[10pt,reqno]{amsart}
\usepackage{amssymb,mathrsfs}
\usepackage[margin=0.8in]{geometry}
\usepackage{ifthen}
\usepackage{colortbl}
\usepackage{cases,appendix}
\definecolor{black}{rgb}{0.0, 0.0, 0.0}
\definecolor{red}{rgb}{1.0, 0.5, 0.5}

\provideboolean{shownotes} 
\setboolean{shownotes}{true} 
\newcommand{\margnote}[1]{
	\ifthenelse{\boolean{shownotes}}%
	{\marginpar{\raggedright\tiny\texttt{#1}}}%
	{}%
}
\newcommand{\hole}[1]{
	\ifthenelse{\boolean{shownotes}}%
	{\begin{center} \fbox{ \rule {.25cm}{0cm} \rule[-.1cm]{0cm}{.4cm}
				\parbox{.85\textwidth}{\begin{center} \texttt{#1}\end{center}} \rule
				{.25cm}{0cm}}\end{center}} {} }

\title[Cucker-Smale flocking particles with multiplicative noises]{Cucker-Smale flocking particles with multiplicative noises: stochastic mean-field limit and phase transition}

\author[Choi]{Young-Pil Choi}
\address[Young-Pil Choi]{\newline Department of Mathematics and Institute of Applied Mathematics
	\newline Inha University, Incheon 402--751, Korea}
\email{ypchoi@inha.ac.kr}

\author[Salem]{Samir Salem}
\address[Samir Salem]{\newline CEREMADE UMR CNRS 7534 \newline
	Universit\'e Paris-Dauphine, Place du Maréchal de Lattre de Tassigny 75775, Paris Cedex 16, France}
\email{salem@ceremade.dauphine.fr}

\numberwithin{equation}{section}

\newtheorem{theorem}{Theorem}[section]
\newtheorem{lemma}{Lemma}[section]

\newtheorem{proposition}{Proposition}[section]
\newtheorem{remark}{Remark}[section]
\newtheorem{definition}{Definition}[section]

\newcommand{\R}{\mathbb R}

\newcommand{\mc}{\mathcal C}

\newcommand{\bq}{\begin{equation}}
\newcommand{\eq}{\end{equation}}

\newcommand{\lt}{\left}
\newcommand{\rt}{\right}
\newcommand{\lal}{\langle}
\newcommand{\ral}{\rangle}

\newcommand{\mt}{\mathcal{T}}

\newcommand{\pp}{\mathcal{P}}

\newcommand{\E}{\mathbb{E}}
\newcommand{\mx}{\mathcal{X}}
\newcommand{\mv}{\mathcal{V}}
\newcommand{\mz}{\mathcal{Z}}

\begin{document}
	\allowdisplaybreaks
	
	\date{\today}
	
	
	\begin{abstract} In this paper, we consider the Cucker-Smale flocking particles which are subject to the same velocity-dependent noise, which exhibits a phase change phenomenon occurs bringing the system from a ``non flocking'' to a ``flocking'' state as the strength of noises decreases. We rigorously show the stochastic mean-field limit from the many-particle Cucker-Smale system with multiplicative noises to the Vlasov-type stochastic partial differential equation as the number of particles goes to infinity. More precisely, we provide a quantitative error estimate between solutions to the stochastic particle system and measure-valued solutions to the expected limiting stochastic partial differential equation by using the Wasserstein distance. For the limiting equation, we construct global-in-time measure-valued solutions and study the stability and large-time behavior showing the convergence of velocities to their mean exponentially fast almost surely.  
	\end{abstract}
	
	\maketitle \centerline{\date}

	
	%
	%
	%
	%
	\section{Introduction}
	In the current work, we are interested in stochastic flocking systems with multiplicative noises in the Stratonovich sense. Let $(\Omega,\mathcal{F},(\mathcal{F}_t)_{t\geq 0},\mathbb{P})$ be a probability space endowed with a filtration $(\mathcal{F}_t)_{t \geq 0}$. Here $\Omega$ is the random set, $\mathbb{P}$ and $\mathcal{F}$ are measure and $\sigma$-algebra on the set, respectively. On that probability space, $(B_t)_{t \geq 0}$ denotes a real-valued Brownian motion. Let $X^i_t \in \R^d$ and $V^i_t \in \R^d$ be position and velocity of $i$-th particle at time $t\geq 0$, respectively, then our main stochastic differential equations read as follows:
	\begin{align}\label{main_sde}
		\begin{aligned}
			&dX^i_t = V^i_t\, dt, \quad i=1,\cdots, N, \quad t > 0,\cr
			&dV^i_t = F[\mu^N_t](X^i_t,V^i_t)\,dt + \sqrt{2\sigma}(\bar V_t - V^i_t) \circ dB_t, \quad \mu_t^N:=\frac1N \sum_{i=1}^N \delta_{(X^i_t,V^i_t)},
		\end{aligned}
	\end{align}
	or, equivalently, in the It\^o form,
	\begin{align}\label{main_sde_ito}
		\begin{aligned}
			&dX^i_t = V^i_t\, dt, \quad 1,\cdots, N, \quad t > 0,\cr
			&dV^i_t = F[\mu^N_t](X^i_t,V^i_t)\,dt - \sigma(\bar V_t - V^i_t)\,dt + \sqrt{2\sigma}(\bar V_t - V^i_t)\, dB_t, 
		\end{aligned}
	\end{align}
	subject to the deterministic initial data $(X^i_0,V^i_0)$, for $i=1,\cdots,N$. Here $\bar V_t$ is an averaged particle velocity, i.e., $\bar V_t := \frac1N \sum_{j=1}^N V^j_t$ and $F[\mu]$ represents a velocity alignment force given by
	\[
	F[\mu](x,v) := \int_{\R^d \times \R^d} \psi(|x-y|)(w-v)\, \mu(dy, dw) \quad \mbox{for} \quad \mu \in \pp(\R^d \times \R^d),
	\] 
	where $\psi: \R_+ \to \R_+$ called a communication weight, which is in general non-increasing function. Note that the stochastic particle system \eqref{main_sde} has locally Lipschitz coefficients, thus the system \eqref{main_sde} has a strong solutions and pathwise uniqueness holds, see \cite[Theorem 3.2]{Dur}. 
	
	When there is no noise, i.e., $\sigma = 0$, the stochastic particle system \eqref{main_sde} is reduced to the Cucker-Smale model \cite{CS07}. We refer to \cite{CCP, CHL} for a recent overview of Cucker-Smale and its variants. The system \eqref{main_sde} is proposed in \cite{AH} by taking into account uniform randomness in the communication weight function. More precisely, the system \eqref{main_sde} can be derived from the original Cucker-Smale model by replacing $\psi$ with $\psi + \sqrt{2\sigma}\eta_t$, where $\eta_t$ is $d$-dimensional Gaussian white noise. In \cite{AH}, a flocking estimate showing the relative positions are uniformly bounded in time and relative velocities converge to zero as time goes to infinity almost surely is obtained. Later, in \cite{TLY}, the phase change phenomenon from non flocking to flocking states in \eqref{main_sde} is observed by considering the convergence of relative velocities in the $L^2$-norm. For a flocking estimate of the Cucker-Smale model with $N$-independent noise of uniform strength, we refer to \cite{HLL}.  
	
Formal passage to the mean-field limit $N \to \infty$ for the particle system \eqref{main_sde} yields the following stochastic partial differential equation: 
	\begin{align}\label{main_spde}
		\begin{aligned}
			d\mu_t + (v \cdot \nabla_x \mu_t)dt + \lt(\nabla_v \cdot (F[\mu_t]\mu_t)\rt)dt + \sqrt{2\sigma} \nabla_v \cdot \lt((\bar v_t - v)\mu_t\rt) \circ dB_t = 0,
		\end{aligned}
	\end{align}
	or again equivalently, in the It\^o form:
	\begin{align}\label{main_spde_ito}
		\begin{aligned}
			&d\mu_t + (v \cdot \nabla_x \mu_t)dt + \lt(\nabla_v \cdot (F[\mu_t]\mu_t)\rt)dt + \sqrt{2\sigma} \nabla_v \cdot \lt((\bar v_t - v)\mu_t\rt) dB_t \cr
			&\qquad \qquad = \sigma \nabla_v \cdot ((\bar v_t - v)\nabla_v \cdot ((\bar v_t - v)\mu_t))dt,
		\end{aligned}\end{align}
		where $\bar v_t := \int_{\R^d \times \R^d} v\, \mu_t(dx,dv)$. The limiting equation \eqref{main_spde} is indeed stochastic in this case as shown for instance in \cite{Cog}, where a system of interacting particles are subject to the same space-dependent noise is discussed. This is due to the fact that the noise which drives the motion of each particle in \eqref{main_sde_ito} is the same. In classical McKean-Vlasov particle system, the noise seen by each particle are independent from each other \cite[Theorem 1.1]{Szn}, and the limiting equation becomes a deterministic diffusion equation. This result can be classically proved by coupling the $N$-particle system with independent initial condition to the $N$ independent copies of the nonlinear particle. It is worth emphasizing that the independence of noises in the system is important in that coupling method, see \cite{Szn} for more details on that. 
		
		The first purpose of this paper is to establish the global existence and uniqueness of measure-valued solutions to the stochastic partial differential equation \eqref{main_spde}, and the rigorous analysis of the stochastic mean-field limit of the system \eqref{main_sde}. As pointed out in \cite{Cog}, the equation \eqref{main_spde} can be understood as a standard transport PDE as the random $\omega \in \Omega$ is fixed. The empirical measure $(\mu_t^N)_{t\geq 0}$ associated to the stochastic particle system \eqref{main_sde} solves the stochastic partial differential equation \eqref{main_spde} for any finite $N$, see Section \ref{sec_weak} below for details. This enables us to take a strategy based on weak-weak/weak-strong stability estimates used for deterministic transport type equations \cite{CCH, CCHS, Dobr} for fixed $\omega \in \Omega$. More precisely, if $(\mu_t)_{t\geq 0}$ and $(\nu_t)_{t\geq 0}$ are two solutions to \eqref{main_spde} for the respective initial data $\mu_0$ and $\nu_0\in \mathcal{P}_2(\R^d\times \R^d)$ in the space $\mathcal{C}([0,T];\mathcal{P}_2(\R^d\times \R^d))$, we need to establish some (local in time) inequalities of the type:
		\bq
		\label{eq:exp}
		\sup_{t\in [0,T]}\E\lt[W_2(\mu_t,\nu_t) \rt] \leq C_T \E\lt[W_2(\mu_0,\nu_0)\rt], 
		\eq
		where $C_T$ is a nonnegative constant depending on the time and other parameters of the problem, or a weaker version
		\bq
		\label{eq:as}
		\sup_{t\in [0,T]}W_2(\mu_t,\nu_t) \leq C_T W_2(\mu_0,\nu_0) \quad \mbox{almost surely.}
		\eq	
Here $C_T$ is a nonnegative almost surely finite random variable depending on the time and other parameters of the problem. Here $W_2$ denotes the Wasserstein distance of order $2$ defined by
		\[
		W_2^2(\mu,\nu):=\inf_{\xi \in \Gamma(\mu,\nu)}\lt(\int_{\R^d \times \R^d} |x-y|^2 \xi(dx,dv) \rt) = \inf_{(X \sim \mu, Y \sim \nu)}\E[|X-Y|^2],
		\]
		where $\Gamma(\mu,\nu)$ is the set of all probability measures on $\R^d \times \R^d$ with first and second marginals $\mu$ and $\nu$, respectively, and $(X,Y)$ are all possible couples of random variables with $\mu$ and $\nu$ as respective laws. Compared to the classical case of globally Lipschitz and bounded potentials, the force fields in \eqref{main_sde} are only locally Lipschitz and bounded in velocity and the result by Dobrushin \cite{Dobr} cannot be directly applied. Note that a classical feature of the Cucker-Smale equation with nonnegative communication weight is to keep the speed of particle velocity bounded by the maximal speed at initial state. On the other hand, in the presence of diffusion, that maximum principle does not hold since the Brownian motion can make the velocities as high as wanted with some non-zero probability. In \cite{2jos}, similar Newtonian types of equations with independent standard Brownian motions, which have locally Lipschitz potentials, are considered, and the high speed of particle velocities are controlled by imposing the exponential moments bound. However, in our case for the stochastic transport PDE, we can obtain a $\mathbb{P}$-almost sure propagation of the compact support in velocity if the initial data is compactly supported in velocity. This only gives that the force fields are Lipschitz and bounded $\mathbb{P}$-almost surely, thus we can have a similar inequality as \eqref{eq:as}, but not the type of \eqref{eq:exp} since the Lipschitz constant of force fields is a random variable which does not have any exponential moments, see Proposition \ref{lem:stab}. This stability estimate enables us to approximate a solution $\mu_t$ to the equation \eqref{main_spde} by the empirical measure $\mu^N_t$ associated to the particle system \eqref{main_sde}, and in fact, this provides the stochastic mean-field limit. We remark that the mean-field limit of the particle system \eqref{main_sde} is studied in \cite{HJNXZ}, and a Fokker-Planck type equation is derived as the corresponding mean-field equation. However, that corresponds to the Cucker-Smale model with $N$-independent Brownian motions, i.e., adding $\sqrt{2\sigma}(\bar V_t - V^i_t) dB^i_t$, not the dependent Brownian motion appeared in \eqref{main_sde}.
		
		Our second goal in this paper is to discuss the phase change phenomenon in the limiting stochastic kinetic equation \eqref{main_spde} showing the transition from non flocking to flocking states as the strength of noises decreases. We notice that flocking behavior of solutions implies the concentration of velocities of particles, i.e., formation of a Dirac delta in velocity, see Remark \ref{rmk_ltt}. Thus it is natural to consider measure-valued solutions in our notion of solutions for the time-asymptotic behavior of solutions. Since the empirical measures associated to the particle system \eqref{main_sde} well approximate the measure-valued solutions to the stochastic partial differential equation \eqref{main_spde}, see Proposition \ref{lem:stab}, we can easily extend the result of phase change phenomenon at particle level to the infinite-dimensional one.
		
		Before stating our main results, we first introduce a notion of measure-valued solutions to the kinetic system \eqref{main_spde}. For this, we use a standard notation:
		\[
		\lal\nu, \phi \ral := \int_{\R^d \times \R^d} \phi(x,v)\,\nu(dx,dv), \quad \mbox{for} \quad \nu \in \pp_2(\R^d \times \R^d).
		\]
		\begin{definition}\label{def_weak} A family $\{\mu_t(w)\,:\,t\geq0,\,\omega\in \Omega \}$ of random probability measures taking value in $\pp_2(\R^d \times \R^d)$ is a measure-valued solution of the equation \eqref{main_spde} if 
			\begin{itemize}
				\item[(i)] $\mu_t$ is weakly continuous: for all $\phi \in \mc^2_c(\R^d \times \R^d)$, $\lal \mu_t, \phi\ral$ is an adapted process with a continuous version.
				\item[(ii)] $\mu_t$ satisfies the stochastic integral equation: for all $\phi \in \mc^2_c(\R^d \times \R^d)$,
				\[
				\lal\mu_t,\phi \ral = \lal \mu_0, \phi\ral + \int_0^t \lal \mu_s, v \cdot\nabla_x \phi + F[\mu_s]\cdot\nabla_v \phi\ral \,ds + \sqrt{2\sigma}\int_0^t \lal \mu_s, (\bar v_s - v)\cdot \nabla_v \phi \ral \circ dB_s.
				\]
			\end{itemize}
		\end{definition}
		\begin{remark} The weak formulation in Definition \ref{def_weak} can be rewritten as 
			$$\begin{aligned}
			\lal\mu_t,\phi \ral &= \lal \mu_0, \phi\ral + \int_0^t \lal \mu_s, v \cdot\nabla_x \phi + (F[\mu_s]- \sigma(\bar v_s - v))\cdot\nabla_v \phi \ral \,ds \cr
			&\qquad + \sqrt{2\sigma}\int_0^t \lal \mu_s, (\bar v_s - v)\cdot \nabla_v \phi \ral \,dB_s
			+ \sigma \int_0^t \lal \mu_s, (\bar v_s - v )\otimes (\bar v_s - v) : \nabla^2_v \phi \ral\,ds.
			\end{aligned}$$
		\end{remark}
		
		We now state our first result on the global existence and uniqueness of measure-valued solutions to the stochastic partial differential equation \eqref{main_spde}.
		\begin{theorem}\label{thm_smf}
			Let $\mu_0\in \mathcal{P}_2(\R^d \times \R^d)$ be compactly supported in velocity and $T>0$. Suppose that the communication weight $\psi \in \mc^1_b(\R_+)$. Then there exists at most one measure-valued solution to equation \eqref{main_sde_ito} $\mu_.\in \mathcal{C}\lt([0,T],  \mathcal{P}_2(\R^d \times \R^d)\rt)$ in the sense of Definition \ref{def_weak}, which is almost surely compactly supported in velocity. Moreover, $\mu_t$ is determined as the push-forward of the initial density through the stochastic flow map generated by the local Lipschitz field $(v, F[\mu_t] + \sigma (v  - \bar v_t) - \sqrt{2\sigma} (v - \bar v_t)dB_t/dt)$ in phase space. Furthermore, if $\mu$ and $\tilde \mu$ are two such solutions to the equation \eqref{main_sde_ito} with compactly supported initial data $\mu_0$ and $\tilde\mu_0$ in velocity, we have
			\[
			W_2(\mu_t, \tilde\mu_t) \leq CW_2(\mu_0, \tilde\mu_0)e^{C(1 + W_2(\mu_0, \tilde\mu_0))},
			\]
			for $t \in [0,T]$ almost surely, where the constant $C$ depends only on $\psi,T,\sigma$, $\sup_{t \in [0,T]}|B_t|$, and the support in velocity of $\mu_0$ and $\tilde \mu_0$.
		\end{theorem}
		
		\begin{remark} As mentioned before, the empirical measure $\mu^N_t = \frac1N\sum_{i=1}^N \delta_{(X^i_t,V^i_t)}$ associated to the particle system \eqref{main_sde} is the solution to the stochastic partial differential equation \eqref{main_spde} in the sense of Definition \ref{def_weak}, see Section \ref{sec_weak}. Thus it follows from the stability estimate in Theorem \ref{thm_smf} that
			\[
			\sup_{0 \leq t \leq T}W_2(\mu_t, \mu^N_t) \leq CW_2(\mu_0, \mu^N_0)e^{C(1 + W_2(\mu_0, \mu^N_0))},
			\]
			where $C$ is a random variable independent of $N$. Thus if $W_2(\mu_0, \mu^N_0) \to 0$ as $N \to \infty$, we have
			\[
			\sup_{0 \leq t \leq T}W_2(\mu_t, \mu^N_t) \quad \mbox{as} \quad N \to \infty, \quad \mbox{almost surely}.
			\] 
			Note that we can construct the initial atomic measures $\mu^N_0$ approximating the initial data $\mu_0$ such that $W_2(\mu_0, \mu^N_0) \to 0$ as $N \to \infty$ in the standard way: we define a regular mesh of size $1/N$ and approximate $\mu_0$ by a sum of Dirac masses $\mu^N_0$ located at the center of the regular cells such that the mass at each particle is exactly equals to the mass of $\mu_0$ contained in the associated cell. Then we get $W_2(\mu_0, \mu^N_0) \sim 1/N \to 0$ as $N \to \infty$.
		\end{remark}
		
		Our second result on the phase change phenomenon from a ``non flocking'' to a ``flocking'' state depending on the strength of noises is presented below. In order to state our theorem, we need to introduce a variance functional of the stochastic particle velocity fluctuation around $\bar v_0$:
		\[
		E[\mu_t] := \int_{\R^d \times \R^d} |\bar v_0 - v|^2 \mu_t(dx,dv).
		\]
		For the sake of notational simplicity, we denote by $E_t := E[\mu_t]$. 
		\begin{theorem}\label{thm_spde_f}Let $\mu_t$ be a measure-valued solution for the equation \eqref{main_spde}. Suppose that the communication weight function $\psi$ satisfies $0 < \psi_m \leq \psi(s) \leq \psi_M$ for $s \in \R_+$. Then we have
			\[
			\E[E_0]e^{-2(\psi_M - 2\sigma) t} \leq \E [E_t] \leq \E[E_0]e^{-2(\psi_m - 2\sigma) t} \quad t \geq 0.
			\]
			This subsequently implies
			\[
			\lim_{t \to \infty}\E[E_t] = \left\{ \begin{array}{ll}
			0 & \textrm{if $\psi_m > 2\sigma$,}\\[2mm]
			\infty & \textrm{if $\psi_M < 2\sigma$.}
			\end{array} \right.
			\]
		\end{theorem}
		\begin{remark}\label{rmk_ltt} It follows from Theorem \ref{thm_spde_f} that 
			\[
			W_1(\mu_t(x,v), \rho_t(x) \delta_{\bar v_0}(v)) \to 0 \quad \mbox{as} \quad t \to \infty \quad \mbox{in probability},
			\]
			where $W_1$ denotes the Wasserstein distance of order $1$ and $\rho_t$ is the random spatial probability, i.e., $\rho_t(x) := \int_{\R^d} \mu_t(dv)$. Indeed, for any bounded Lipschitz function $\phi$, we find
			$$\begin{aligned}
			&\lt|\int_{\R^d \times \R^d} \phi(x,v) \mu_t(dx, dv) - \int_{\R^d \times \R^d} \phi(x,v) \rho_t(dx)\delta_{\bar v_0}(dv) \rt|\cr
			&\quad = \lt|\int_{\R^d \times \R^d} \lt( \phi(x,v) - \phi(x,\bar v_0)\rt)\mu_t(dx, dv)\rt| \leq \|\phi\|_{Lip}\int_{\R^d \times \R^d} |v - v_0| \,\mu_t(dx, dv) \leq \|\phi\|_{Lip}E_t^{1/2} \to 0, 
			\end{aligned}$$
			as $t \to \infty$ in probability, due to Theorem \ref{thm_spde_f}. This, together with the fact that $W_1$ is equivalent to the bounded Lipschitz distance, concludes the desired result.
		\end{remark}
		\begin{remark}\label{rmk_without} We can obtain the convergence of the variance functional $E_t$ without taking the expectation. More precisely, we find the following almost surely convergence when $\psi(s) \geq \psi_m > 0$:
			\[
			\int_{\R^d \times \R^d} |\bar v_0 - v|^2 \mu_t(dx, dv) \to 0 \quad \mbox{as} \quad t \to \infty, \quad a.s.,
			\]
			at least exponentially fast. Find the details of the proof in Proposition \ref{prop_without}.
		\end{remark}
		
		The rest of the paper is organized as follows. In Section \ref{sec_pre}, we show that the empirical measures associated to the particle system \eqref{main_sde} are solutions to the equation \eqref{main_spde} in the sense of Definition \ref{def_weak}. We also present a stochastic Gronwall type inequality which will be used later for the almost surely bound estimate of compact support of solutions in velocity. Using that support bound estimate in velocity together with the weak stability estimate, we provide details on the proof of Theorem \ref{thm_smf} in Section \ref{sec_well}. Finally, we show the phase change phenomenon of the stochastic partial differential equation \eqref{main_spde}, which proves Theorem \ref{thm_spde_f}, in Section \ref{sec_lt}.
		
		%
		%
		%
		%
		
		\section{Preliminaries}\label{sec_pre}
		
		\subsection{It\^o's formula}\label{sec_weak} In this part, we show that the empirical measures associated to the stochastic particle system \eqref{main_sde_ito} are weak solutions to the stochastic partial differential equation by employing the It\^o's formula. For the sake of mathematical simplicity, we work in the corresponding It\^o form \eqref{main_sde_ito}. We want to emphasize that this observation implies the limiting system cannot be deterministic since the empirical measures are stochastic for any $N$. 
		
		For $\phi \in \mc^2_b(\R^d \times \R^d)$, if we apply for It\^o's formula to the system \eqref{main_sde_ito}, then we obtain
		$$\begin{aligned}
		\phi(X^i_t,V^i_t) &= \phi(X^i_0,V^i_0) + \int_0^t \nabla_x \phi(X^i_s,V^i_s) \cdot V^i_s\,ds + \int_0^t \nabla_v \phi(X^i_s,V^i_s) \cdot F[\mu^N_s](X^i_s,V^i_s)\,ds\cr
		&\quad -\sigma\int_0^t \nabla_v \phi(X^i_s,V^i_s)\cdot(\bar V_s - V^i_s)\,ds+ \sqrt{2\sigma}\int_0^t \nabla_v \phi(X^i_s,V^i_s)\cdot(\bar V_s - V^i_s)\,dB_s \cr
		&\quad + \sigma\int_0^t (\bar V_s - V^i_s)\otimes(\bar V_s - V^i_s) : \nabla_v^2 \phi(X^i_s,V^i_s) \,ds.
		\end{aligned}$$
		Averaging the above equation over $i=1,\cdots,N$ deduces
		$$\begin{aligned}
		\int_{\R^d \times \R^d} \phi(x,v)\mu^N_t(dx,dv) 
		&= \int_{\R^d \times \R^d} \phi(x,v)\mu^N_0(dx,dv) + \int_0^t \int_{\R^d \times \R^d} \nabla_x \phi(x,v) \cdot v \,\mu^N_s(dx,dv)\,ds\cr
		&\quad + \int_0^t \int_{\R^d \times \R^d} \nabla_v \phi(x,v) \cdot F[\mu^N_s](x,v)\, \mu^N_s(dx,dv)\,ds\cr
		&\quad - \sigma\int_0^t \int_{\R^d \times \R^d} \nabla_v \phi(x,v) \cdot (\bar V_s - v)\,\mu^N_s(dx,dv)\,ds\cr
		&\quad + \sqrt{2\sigma}\int_0^t \int_{\R^d \times \R^d} \nabla_v \phi(x,v) \cdot (\bar V_s - v)\,\mu^N_s(dx,dv)\,dB_s\cr
		&\quad + \sigma\int_0^t \lt((\bar V_s - v) \otimes (\bar V_s - v)\rt) : \nabla_v^2\phi(x,v)\, \mu^N_s(dx,dv)\,ds.
		\end{aligned}$$
		Note that $\bar V_t = \int_{\R^d \times \R^d} v \mu^N_t(dx,dv)$. We also easily check that 
		$$\begin{aligned}
		&\int_0^t \int_{\R^d \times \R^d} \phi(x,v) \nabla_v \cdot \lt((\bar v_s - v)\nabla_v \cdot ((\bar v_s - v)\mu_s(dx,dv))\rt)ds\cr
		&\quad = - \int_0^t \int_{\R^d \times \R^d} \nabla_v \phi(x,v) \cdot ((\bar v_s - v)\nabla_v \cdot ((\bar v_s - v)\mu_s(dx,dv)))\,ds\cr
		&\quad =  \int_0^t \int_{\R^d \times \R^d} \nabla_v \phi(x,v) \cdot (\bar v_s - v)\mu_s(dx,dv)\,ds  - \int_0^t \int_{\R^d \times \R^d} \nabla_v \phi(x,v) \cdot \nabla_v \cdot ((\bar v_s - v)\otimes (\bar v_s - v)\mu_s(dx,dv))\,ds\cr
		&\quad =  \int_0^t \int_{\R^d \times \R^d} \nabla_v \phi(x,v) \cdot (\bar v_s - v)\mu_s(dx,dv)\,ds + \int_0^t \int_{\R^d \times \R^d} \nabla_v^2 \phi(x,v) : ((\bar v_s - v)\otimes (\bar v_s - v))\mu_s(dx,dv)\,ds,
		\end{aligned}$$
		where we used
		\[
		\nabla_v \cdot ((\bar v_s - v)\otimes (\bar v_s - v)\mu_s) = (\bar v_s - v)\mu_s + (\bar v_s - v)\nabla_v \cdot  ((\bar v_s - v)\mu_s).
		\]
		This yields that $\mu^N_t$ associated to \eqref{main_sde_ito} satisfies the weak formulation in Definition \ref{def_weak}. Furthermore, $\lal \mu^N_t, \phi\ral$ with $\phi \in \mc^2_b(\R^d \times \R^d)$ is $\mathcal{F}_t$-adapted since the processes $(X^i_t, V^i_t)_{i=1,\cdots,N}$ are solutions to the system \eqref{main_sde}. Thus they are adapted and continuous and this gives that $\mu^N_t$ is a weak solution to the system \eqref{main_spde_ito} in the sense of Definition \ref{def_weak}.

		\subsection{Stochastic Gronwall type inequality}
		In this subsection, we provide a stochastic Gronwall type inequality which will be crucially used in this paper.

		\begin{lemma}\label{lem_gro} 
			Let $X_t$ be a real value process satisfaying
			\[
			X_t= X_0+\int_0^t \beta_s \,ds -c_2\int_0^t X_s\, dB_s \quad \mbox{for} \quad  t \geq 0.
			\]
			Furthermore, we assume that 
			\[
			\beta_s\leq c_1 X_s+A_s \quad \mbox{for} \quad s \geq 0,
			\]
			then it holds 
			\[
			X_t\leq X_0 e^{(c_1-c_2^2/2)t}e^{-c_2B_t}\lt(\int_0^t e^{-(c_1-c_2^2/2)s}e^{c_2B_s} A_s\, ds+1\rt),
			\]
			with probability one.
		\end{lemma}
		\begin{proof}
			Let us denote
			\[
			\bar{Y}_t :=X_0 e^{(c_1-c_2^2/2)t}e^{-c_2B_t}.
			\]
			Using It\^o's rule, we find that $\bar{Y}_t$ solves
			$$\begin{aligned}
			d\bar{Y}_t&=\lt(c_1 - \frac{c_2^2}{2}\rt)X_0 e^{(c_1-c_2^2/2)t}e^{-c_2B_t}dt-c_2X_0 e^{(c_1-c_2^2/2)t}e^{-c_2B_t}+\frac{c_2^2}{2}X_0 e^{(c_1-c_2^2/2)t}e^{-c_2B_t}dB_t\cr
			&=(c_1dt-c_2dB_t)\bar{Y}_t.
			\end{aligned}$$
			We next consider
			\[
			Y_t:=X_0\int_0^t \frac{e^{(c_1-c_2^2/2)t - c_2B_t}}{e^{(c_1-c_2^2/2)s - c_2B_s}} A_s\, ds=:\int_0^t\frac{Z_t}{Z_s}A_s \,ds.
			\]
			Then again by It\^o's rule we get
			\[
			dY_t=X_0\lt(\int_0^t\frac{1}{Z_s}A_s\,ds\rt)dZ_t+X_0Z_t\frac{1}{Z_t}A_t\,dt+X_0^2\frac{1}{Z_t}A_t\,dtdZ_t.
			\]
			On the other hand, $Z_t$ satisfies
			$$\begin{aligned}
			dZ_t &=\lt(c_1 - \frac{c_2^2}{2}\rt)e^{(c_1-c_2^2/2)t}e^{-c_2B_t}dt+e^{(c_1-c_2^2/2)t}\lt(-c_2e^{-c_2B_t}dB_t+c_2^2/2e^{-c_2B_t}dt\rt)=(c_1dt-c_2dB_t)Z_t.
			\end{aligned}$$
			This yields that $H_t=Y_t+\bar{Y}_t$ solves
			\[
			dH_t=c_1 H_t dt-c_2H_tdB_t +A_tdt,
			\]
			with the initial condition $H_0=X_0$. This together with \cite[Theorem 1.1, p 437]{IkWa} concludes 
			\[
			X_t\leq H_t \quad \mbox{for} \quad t \geq 0,
			\]
			almost surely. This completes the proof.
		\end{proof}	
		
		%
		%
		%
		%
		\section{Well-posedness of stochastic PDE: Proof of Theorem \ref{thm_smf}}\label{sec_well}
		In this section we establish the global well-posedness of the equation \eqref{main_sde_ito}. Note that we already observed that the empiricial measure associated to the particle system \eqref{main_sde} is a weak solution to \eqref{main_sde_ito} in the previous section. 
		
		Let us consider two solutions $(\mu_t)_{t\geq 0}$ and $(\tilde{\mu}_t)_{t\geq 
			0}$ to the stochastic partial differential equation \eqref{main_spde_ito} in the sense of Definition \ref{def_weak} with the initial data $\mu_0, \tilde \mu_0\in \mathcal{P}_2(\R^d \times \R^d)$, respectively, such that
		\[
		\int_{\R^d \times \R^d}v \mu_0(dx,dv)=\int_{\R^d \times \R^d} v \tilde \mu_0(dx,dv)=0.
		\]
		Note that we can assume the above without loss of generality due to the conservation of momentum. In particular, this implies 
		\[
		\int_{\R^d \times \R^d}v \mu_t(dx,dv)=\int_{\R^d \times \R^d} v \tilde{\mu}_t(dx,dv)=0 \quad \mbox{for} \quad t\geq 0, \quad \mbox{almost surely}.
		\]
		For $(x,v,t) \in \R^d \times \R^d \times [0,T)$, we define $\mz_t(x,v)=(\mx_t(x,v),\mv_t(x,v))$, the stochastic characteristic with the initial data $\mz_0(x,v) = (x,v)$ as:
		\begin{align}\label{sto_char}
			\begin{aligned}
				&\mx_t(x,v) = x+\int_0^t\mv_s(x,v)\, ds, \quad t \geq 0,\cr
				&\mv_t(x,v)= v+\int_0^t F[\mu_s](\mz_s(x,v))\,ds + \sigma\int_0^t \mv_s(x,v)\,ds - \sqrt{2\sigma}\int_0^t \mv_s(x,v)\, dB_s.
			\end{aligned}
		\end{align}
		Note that the above stochastic characteristics are globally well-defined since the velocity alignment force term is locally bounded and Lipschitz. Applying It\^o's rule, it is clear that  
		\[
		\begin{split}
		\int \phi(\mz_t(x,v))\mu_0(dx,dv)&=\int \phi(x,v) \mu_0(dx,dv)+\int_0^t \int \mv_s(x,v)\cdot\nabla_x\phi(\mz_s(x,v))\mu_0(dx,dv)\, ds\\
		&\quad +\int_0^t \int F[\mu_s](\mz_s(x,v))\cdot \nabla_v \phi(\mz_s(x,v))\mu_0(dx,dv) \,ds\\
		&\quad +\int_0^t \sigma \int  \mv_s(x,v)\cdot \nabla_v \phi(\mz_s(x,v))\mu_0(dx,dv)\,ds \\
		&\quad -\int_0^t  \sqrt{2\sigma} \int \mv_s(x,v)\cdot \nabla_v \phi(\mz_s(x,v))\mu_0(dx,dv)\,dB_s\\
		&\quad +\int_0^t \sigma \int\lt|\mv_s(x,v)\rt|^2\Delta_v \phi(\mz_s(x,v))\mu_0(dx,dv)\,ds,
		\end{split}
		\]
		for any $\phi\in \mathcal{C}_c^2(\R^d \times \R^d)$. We now define $\hat{\mu}_t$ by the push-forward of $\mu_0$ by $\mz_t$, i.e., $\hat{\mu}_t:=\mz_t\#\mu_0$. Then this gives that the random probability measure family $\hat{\mu}_t$ solves the following linear stochastic PDE:
		\[
		d\hat{\mu}_t+\lt(v\cdot\nabla_x \hat{\mu}_t\rt)dt+\nabla_v\cdot\lt( \hat{\mu}_tF[\mu_t]\rt)dt-\sqrt{2\sigma}\nabla_v\cdot\lt(\hat{\mu}_t v\rt)dB_t=\sigma \nabla_v\cdot\lt(v\cdot\nabla_v (\hat{\mu}_tv)\rt)dt,
		\]
		with the initial data $\hat{\mu}_0=f_0$. On the other hand, the uniqueness of solutions for that linear stochastic PDE holds, thus we get 
		\[
		\mu_t=\hat{\mu}_t= \mz_t\#\mu_0.
		\]
		In the lemma below, we provide {\it a priori} kinetic energy and velocity support estimates.
		\begin{lemma}	\label{lem:EstVel}
			Let $(\mu_t)_{t\geq 0}$ be a solution of \eqref{main_spde}, and $(\mz_t)_{t\geq 0}$ be the associated stochastic characteristic. Then, it holds
			\[
			\int_{\R^d \times \R^d} |v|^2 \mu_t(dx,dv)=\int_{\R^d \times \R^d} |\mv_t(x,v)|^2 \mu_0(dx,dv)\leq \lt( \int_{\R^d \times \R^d} |v|^2 \mu_0(dx,dv)\rt) e^{-2\sqrt{2\sigma}B_t},
			\]
			with probability one. Moreover, still with probability one, it holds 
			\[
			\lt|\mv_t(x,v)\rt|^2\leq |v|^2e^{\psi_M t-2\sqrt{2\sigma}B_t}\lt(\lt(\int_{\R^d \times \R^d}|w|^2 \mu_0(dy,dw)   \rt)\frac{1}{\psi_M}+1\rt) \quad \mbox{for} \quad (x,v)\in \R^d \times \R^d.
			\]
		\end{lemma}
		\begin{proof} It follows from \eqref{sto_char} that
			\begin{align}\label{eq:SPDEm2}
				\begin{aligned}
					\lt|\mv_t(x,v)\rt|^2 &= |v|^2 +2\int_0^t \left\langle \mv_s(x,v), F[\mu_s](\mz_s(x,v))\right\rangle \, ds+4\sigma \int_0^t \lt|\mv_s(x,v)\rt|^2 \, ds  - 2\sqrt{2\sigma}\int_0^t\lt|\mv_s(x,v)\rt|^2\, dB_s.
				\end{aligned}
			\end{align}
			Integrating \eqref{eq:SPDEm2} with respect to $\mu_0(dx,dv)$, we find 
			$$\begin{aligned}
			&\int_{\R^d \times \R^d}\lt|\mv_t(x,v)\rt|^2 \mu_0(dx,dv) \cr
			&\quad = \int_{\R^d \times \R^d}|v|^2 \mu_0(dx,dv) -2\int_0^t \int_{\R^{2d} \times \R^{2d}}\psi(\mx_t(x,v)-\mx_t(y,w) )\lt| \mv_t(x,v)-\mv_t(y,w)\rt|^2 \mu_0(dy,dw) \mu_0(dx,dv)\,ds\\ 
			&\qquad +4\sigma \int_0^t \int_{\R^d \times \R^d}\lt|\mv_s(x,v)\rt|^2 \mu_0(dx,dv) \, ds  - 2\sqrt{2\sigma}\int_0^t\int_{\R^d \times \R^d}\lt|\mv_s(x,v)\rt|^2 \mu_0(dx,dv)\, dB_s\\
			&\quad \leq \int_{\R^d \times \R^d}|v|^2 \mu_0(dx,dv)+4\sigma \int_0^t \int_{\R^d \times \R^d}\lt|\mv_s(x,v)\rt|^2  \mu_0(dx,dv) \, ds - 2\sqrt{2\sigma}\int_0^t\int_{\R^d \times \R^d}\lt|\mv_s(x,v)\rt|^2 \mu_0(dx,dv)\, dB_s.
			\end{aligned}$$
			This together with Lemma \ref{lem_gro} gives
			\[
			\int_{\R^d \times \R^d}\lt|\mv_t(x,v)\rt|^2 \mu_0(dx,dv)\leq \lt(\int_{\R^d \times \R^d}|v|^2 \mu_0(dx,dv)   \rt)e^{-2\sqrt{2\sigma}B_t}.
			\]
			Coming back to \eqref{eq:SPDEm2}, we obtain
			$$\begin{aligned}
			\lt|\mv_t(x,v)\rt|^2
			&= |v|^2 +2\int_0^t \int_{\R^d \times \R^d} \psi(\mx_t(x,v)-\mx_t(y,w))\left\langle \mv_s(x,v),\mv_s(y,w))-\mv_s(x,v))\right\rangle \mu_0(dy,dw) \, ds\\
			&\quad +4\sigma \int_0^t \lt|\mv_s(x,v)\rt|^2 \, ds  - 2\sqrt{2\sigma}\int_0^t\lt|\mv_s(x,v)\rt|^2dB_s\\
			&\leq |v|^2 +2\int_0^t \int_{\R^d \times \R^d} \psi_M \lt|\left\langle \mv_s(x,v),\mv_s(y,w)\right\rangle\rt| \mu_0(dy,dw) \, ds+4\sigma \int_0^t \lt|\mv_s(x,v)\rt|^2 \, ds \cr
			&\qquad  - 2\sqrt{2\sigma}\int_0^t\lt|\mv_s(x,v)\rt|^2 dB_s\\
			&\leq |v|^2 +\int_0^t \int_{\R^d \times \R^d} \psi_M\lt|\mv_s(y,w)\rt|^2 \mu_0(dy,dw) \, ds+(4\sigma+\psi_M) \int_0^t \lt|\mv_s(x,v)\rt|^2 \, ds  \cr
			&\quad - 2\sqrt{2\sigma}\int_0^t\lt|\mv_s(x,v)\rt|^2 dB_s\\
			&\leq |v|^2 +(4\sigma+\psi_M) \int_0^t \lt|\mv_s(x,v)\rt|^2 ds  - 2\sqrt{2\sigma}\int_0^t\lt|\mv_s(x,v)\rt|^2 dB_s \cr
			&\quad +\int_0^t \lt(\int_{\R^d \times \R^d}|w|^2 \mu_0(dy,dw)   \rt)e^{-2\sqrt{2\sigma}B_s} \, ds.
			\end{aligned}$$
			Hence, by Lemma \ref{lem_gro}, we have
			\[
			\lt|\mv_t(x,v)\rt|^2\leq |v|^2e^{\psi_M t-2\sqrt{2\sigma}B_t}\lt(\lt(\int_{\R^d \times \R^d}|w|^2\mu_0(dy,dw)   \rt)\frac{1-e^{-\psi_Mt}}{\psi_M}+1\rt),
			\]
			and this concludes the desired result.
		\end{proof}
		\begin{remark}\label{rem:comsup}If the initial data $\mu_0$ is compactly supported in velocity, i.e., $supp_v(\mu_0):= \{ v\in \R^d: \mu_0(x,v) \neq 0 \}  \subseteq B(0,R)$ for some $R >0$, then we have
			\[
			\lt|\mv_t(x,v)\rt|^2 \leq R^2e^{\psi_M t-2\sqrt{2\sigma}B_t}\lt(R^2\frac{1-e^{-\psi_Mt}}{\psi_M}+1\rt) \quad \mbox{for} \quad (x,v) \in \R^d \times \R^d.
			\]
			This yields that the support of $\mu_t$ in velocity, $supp_v(\mu_t):= \{v \in \R^d : \mu_t(x,v) \neq 0\}$, is almost surely bounded for $t \in [0,T]$.
		\end{remark}
		We now provide the stability estimate of solutions in $2$-Wasserstein distance in the proposition below. 
		\begin{proposition} \label{lem:stab}
			Let $\mu_t$, $\tilde{\mu}_t$ be solutions to the equation \eqref{main_spde} with compactly supported initial data $\mu_0,\tilde \mu_0\in \mathcal{P}_2(\R^d \times \R^d)$ in velocity, respectively. Then there exists an almost surely finite random variable $C_T$ depends only on $\psi,T,\sigma$, $\sup_{t \in [0,T]}|B_t|$, and the support in velocity of $\mu_0$ and $\tilde \mu_0$, such that
			\[
			W_2(\mu_t,\tilde{\mu}_t) \leq C_T W_2(\mu_0, \tilde \mu_0) e^{2T + C_TW_2(f_0, g_0)},
			\]
			for $t\in [0,T]$, $\mathbb{P}$-almost surely.
		\end{proposition}
		\begin{proof} We first choose an optimal transport map $\mt^0(x,v) = (\mt^0_1(x,v), \mt^0_2(x,v))$ between the initial datum $\mu_0$ and $\tilde \mu_0$ with respect to 2-Wasserstein distance $W_2$, i.e., $\tilde \mu_0 = \mt^0 \# \mu_0$ and
			\[
			W_2 (\mu_0, \tilde\mu_0) = \int_{\R^d \times \R^d} | (x,v) - \mt^0(x,v)|^2 \mu_0(dx,dv).
			\]
			We now consider stochastic characteristics $\mz_t = (\mx_t, \mv_t)$ and $\tilde \mz_t = (\tilde \mx_t, \tilde \mv_t)$ defined as in \eqref{sto_char} associated to the solutions $\mu_t$ and $\tilde\mu_t$, respectively. Then defining $\mt^t := \tilde \mz_t \circ \mt^0 \circ \mz_{-t}$ for $t \in [0,T]$, we find $\mt^t \# \mu_t = \tilde\mu_t$ and
			$$\begin{aligned}
			W^2_2(\mu_t,\tilde\mu_t) &\leq \int_{\R^d \times \R^d} \lt|(x,v) - \mt^t(x,v)\rt|^2 \mu_t(dx,dv) = \int_{\R^d \times \R^d} \lt| \mz_t(x,v) - \tilde\mz_t(\mt^0(x,v))\rt|^2 \mu_0(dx,dv),
			\end{aligned}$$
			due to $\mt^t \circ \mz_t = \tilde\mz_t \circ \mt^0$. Applying It\^o's lemma yields
			\begin{align}\label{est_diff_t2}
				\begin{aligned}
					&\lt|\mv_t(x,v)-\tilde{\mv}_t (\mathcal{T}^0(x,v))\rt|^2\cr
					&\quad = |v-\mathcal{T}^0_2(x,v)|^2 +2\int_0^t \left\langle \mv_s(x,v)-\tilde{\mv}_s (\mathcal{T}^0(x,v)), F[\mu_s](\mz_s(x,v)) - F[\tilde\mu_s](\tilde\mz_s(\mt^0(x,v)))\right\rangle ds\\
					&\qquad +4\sigma \int_0^t \lt|\mv_s(x,v)-\tilde{\mv}_s (\mathcal{T}^0(x,v))\rt|^2 ds - 2\sqrt{2\sigma}\int_0^t\lt|\mv_s(x,v)-\tilde{\mv}_s (\mathcal{T}^0(x,v))\rt|^2 dB_s.
				\end{aligned}
			\end{align}
			Let us denote  
			\[
			P_t:=\int_{\R^d \times \R^d}\lt|\mx_t(x,v)-\tilde{\mx}_t (\mathcal{T}^0(x,v))\rt|^2 \mu_0(dx,dv)
			\]
			and
			\[
			Q_t :=\int_{\R^d \times \R^d}\lt|\mv_t(x,v)-\tilde{\mv}_t (\mathcal{T}^0(x,v))\rt|^2 \mu_0(dx,dv).
			\]
			Then it follows from \eqref{est_diff_t2} that
			\[
			Q_t = Q_0+4\sigma \int_0^t Q_s \, ds - 2\sqrt{2\sigma}\int_0^tQ_s\, dB_s +2\int_0^t I_s\,ds,
			\]
			where 
			\[
			I_s:= \int_{\R^d \times \R^d} \left\langle \mv_s(x,v)-\tilde{\mv}_s (\mathcal{T}^0(x,v)), F[\mu_s](\mz_s(x,v)) - F[\tilde\mu_s](\tilde\mz_s(\mt^0(x,v)))\right\rangle  \mu_0(dx,dv).
			\]
			Note that
			$$\begin{aligned}
			F[\mu_s](\mz_s(x,v)) &= \int_{\R^d \times \R^d} \psi(|\mx_s(x,v) - y|)(w - \mv_s(x,v))\mu_s(dy,dw)\cr
			&=\int_{\R^d \times \R^d} \psi(|\mx_s(x,v) - \mx_s(y,w)|)(\mv_s(y,w) - \mv_s(x,v))\,\mu_0(dy,dw),
			\end{aligned}$$
			due to $\mu_s = \mz_s \# \mu_0$. Similarly, we also get
			$$\begin{aligned}
			&F[\tilde\mu_s](\tilde\mz_s(\mt^0(x,v))) \cr
			&\quad = \int_{\R^d \times \R^d} \psi(|\tilde\mx_s(\mt^0(x,v))) - \tilde\mx_s(\mt^0(y,w)))|) (\tilde\mv_s(\mt^0(y,w)) - \tilde\mv_s(\mt^0(x,v)))\,\mu_0(dy,dw).
			\end{aligned}$$
			This gives
			$$\begin{aligned}
			&F[\mu_s](\mz_s(x,v)) - F[\tilde\mu_s](\tilde\mz_s(\mt^0(x,v)))\cr
			&\quad = \int_{\R^d \times \R^d} \lt(\psi(|\mx_s(x,v) - \mx_s(y,w)|) - \psi(|\tilde\mx_s(\mt^0(x,v))) - \tilde\mx_s(\mt^0(y,w)))|) \rt)(\mv_s(y,w) - \mv_s(x,v))\,\mu_0(dy,dw)\cr
			&\qquad + \int_{\R^d \times \R^d} \psi(|\tilde\mx_s(\mt^0(x,v))) - \tilde\mx_s(\mt^0(y,w)))|) \lt(\mv_s(y,w) - \mv_s(x,v) - \tilde\mv_s(\mt^0(y,w)) + \tilde\mv_s(\mt^0(x,v)) \rt)\mu_0(dy,dw)\cr
			&\quad=:J_s^1 + J_s^2.
			\end{aligned}$$
			Using these newly defined terms, we split $I_s$ into two terms:
			\[
			I_s = \int_{\R^d \times \R^d} \lt\lal \mv_s(x,v) - \tilde\mv_s(\mt^0(x,v)), J_s^1 + J_s^2 \rt\ral \mu_0(dx,dv) =: I_s^1 + I_s^2,
			\]
			and here $I_s^1$ can be estimated as follows.
			$$\begin{aligned}
			I_s^1 &\leq \|\psi\|_{Lip} H_s \int_{\R^d \times \R^d} \lt|\mv_s(x,v) - \tilde\mv_s(\mt^0(x,v)) \rt|\lt|\mx_s(x,v) - \tilde\mx_s(\mt^0(x,v)) \rt|\mu_0(dx,dv)\cr
			&\quad + \|\psi\|_{Lip}H_s \int_{\R^d \times \R^d} \lt|\mv_s(x,v) - \tilde\mv_s(\mt^0(x,v)) \rt|\lt|\mx_s(y,w) - \tilde\mx_s(\mt^0(y,w)) \rt|\mu_0(dx,dv) \mu_0(dy,dw)\cr
			&\leq 2\|\psi\|_{Lip}H_s \,P_s^{1/2}Q_s^{1/2},
			\end{aligned}$$
			where $H_s$ is given by
			\[
			H_s:= 2R^2e^{\psi_M t-2\sqrt{2\sigma}B_s}\lt(\frac{R^2}{\psi_M}+1\rt),
			\]
			where $R>0$ is chosen such that $supp(\mu_0), supp(\tilde\mu_0) \subseteq B(0,R)$ since $\mu_0$ and $\tilde\mu_0$ are compactly supported in velocity.
			We next estimate $I_s^2$ as
			$$\begin{aligned}
			I_s^2 &\leq \psi_M\int_{\R^{2d} \times \R^{2d}}\lt|\mv_s(x,v) - \tilde\mv_s(\mt^0(x,v))\rt| \lt|\mv_s(y,w)  - \tilde\mv_s(\mt^0(y,w))\rt| \mu_0(dx,dv) \mu_0(dy,dw) \leq \psi_MQ_s.
			\end{aligned}$$
			Combining all the above estimates, we obtain
			\[
			Q_t \leq Q_0 + \lt(4\sigma + 2\psi_M\rt)\int_0^t Q_s\,ds - 2\sqrt{2\sigma}\int_0^t Q_s\,ds + 2\|\psi\|_{Lip}\tilde H_T\int_0^t (P_s + Q_s)\,ds
			\]
			for $0 \leq t \leq T$ almost surely, where $\tilde H_T := \sup_{0 \leq t \leq T}H_t$. We now apply Lemma \ref{lem:EstVel} with $c_1 = 4\sigma + 2\psi_M + 2\|\psi\|_{Lip} \tilde H_T$, $c_2 = 2\sqrt{2\sigma}$, and $A_s = 2\|\psi\|_{Lip} \tilde H_T P_s$ to get
			\begin{align}\label{est_q}
				\begin{aligned}
					Q_t &\leq Q_0 e^{c_1 t} e^{-2\sqrt{2\sigma}B_t}\lt(\int_0^t e^{-c_1s} e^{2\sqrt{2\sigma}B_s}2\|\psi\|_{Lip} \tilde H_T P_s\,ds + 1 \rt)\leq C_T Q_0 \lt(\int_0^t P_s\,ds + 1 \rt),
				\end{aligned}
			\end{align}
			where $C_T > 0$ depends only on $\psi, \sigma$, $\sup_{t \in [0,T]}|B_t|$, and $R$. On the other hand, we easily find
			\[
			P_t \leq P_0 + 2\int_0^t P_s\,ds + 2\int_0^t Q_s\,ds
			\]
			This together with \eqref{est_q} yields
			\[
			P_t + Q_t \leq (1 + C_T)(P_0 + Q_0) + (2 + C_T Q_0)\int_0^t \lt(P_s + Q_s \rt)ds,
			\]
			and applying Gronwall's inequality gives
			\[
			P_t + Q_t \leq (1 + C_T)(P_0 + Q_0)e^{2T + C_TQ_0T} \quad \mbox{for} \quad 0 \leq t \leq T \quad \mbox{almost surely}.
			\]
			This completes the proof.
		\end{proof}
		
		We now define the stochastic empirical measure $\mu^N_t$ associated to a solution to the stochastic particle system \eqref{main_sde} as
		\[
		\mu^N_t = \frac{1}{N}\sum_{i=1}^N \delta_{(X^i_t,V^i_t)}
		\]
		with the initial data satisfying
		\[
		W_2(\mu_0^N, \mu_0) \to 0 \quad \mbox{as} \quad N \to \infty,
		\]
		where $\mu_0 \in \pp_2(\R^d \times \R^d)$ is compactly supported in velocity. On the other hand, $\lt(\mathcal{P}_2(\R^d \times \R^d),W_2\rt)$ is a complete metric space, thus the sequence $\mu^N_0$ is Cauchy. As discussed before, $\mu^N_t$ is a solution to the equation \eqref{main_spde} and this together with the stability estimate in Proposition \ref{lem:stab} yields
		\[
		\sup_{t\in [0,T]}W_2(\mu_t^N,\mu_t^{N'})\leq   C_TW_2(\mu_0^N,\mu_0^{N'}) e^{2T + C_TW_2(\mu_0^N,\mu_0^{N'})}.
		\]
		This implies that the $(\mu_.^N)_{N\in \mathbb{N}}$ converges to some $\mu_.\in \mathcal{C}\lt([0,T];\mathcal{P}_2(\R^d \times \R^d)\rt)$ almost surely. Then it remains to prove that $(\mu_t)_{t\in [0,T]}$ solves our main equation \eqref{main_sde_ito}. Note that it follows from Lemma \ref{lem:EstVel} that the solution $\mu_t$ is compactly supported in velocity since the initial data is compactly supported in velocity. 
		
		For $t \in [0,T]$ and $\phi \in \mc^2_b(\R^d \times \R^d)$, we define
		\bq\label{eq:I}
		\begin{aligned}
			I_{\phi,t}(\mu)&:=\lal\mu_t,\phi \ral -\lal \mu_0, \phi\ral - \int_0^t \lal \mu_s, v \cdot\nabla_x \phi + (F[\mu_s]+ \sigma v)\cdot\nabla_v \phi \ral \,ds \cr
			&\qquad + \sqrt{2\sigma}\int_0^t \lal \mu_s,  v\cdot \nabla_v \phi \ral \,dB_s
			- \sigma \int_0^t \lal \mu_s, v \otimes  v : \nabla^2_v \phi \ral\,ds.
		\end{aligned}
		\eq
		We notice that $I_{\phi,t}(\mu^N)=0$ for any $N\in \mathbb{N}$ almost surely. Thus we find
		$$\begin{aligned}
		I_{\phi,t}(\mu)
		&=I_{\phi,t}(\mu)-I_{\phi,t}(\mu^N)\\
		&= \lal\mu_t-\mu_t^N,\phi \ral -\lal \mu_0-\mu_0^N, \phi\ral  - \int_0^t \lal \mu_s-\mu_s^N, v \cdot\nabla_x \phi + \sigma v\cdot\nabla_v \phi \ral \,ds- \sigma \int_0^t \lal \mu_s-\mu_s^N, v \otimes  v : \nabla^2_v \phi \ral\,ds\\
		&\quad  + \sqrt{2\sigma}\int_0^t \lal \mu_s-\mu_s^N,  v\cdot \nabla_v \phi \ral \,dB_s  -\int_0^t \lal \mu_s, F[\mu_s]\cdot\nabla_v \phi \ral- \lal \mu^N_s, F[\mu^N_s]\cdot\nabla_v \phi \ral \,ds\\
		&=: \sum_{i=1}^6 I_i^N.
		\end{aligned}$$
		We then claim that $ \sum_{i=1}^6 I_i^N \to 0$ as $N \to \infty$ in probability to conclude that $\mu_t$ is the solution to equation \eqref{main_spde} in the sense of Definition \ref{def_weak}. \newline
		$\diamond$ Estimate of $ \sum_{i=1}^4 I_i^N$: Note that the terms $I_i^N, 1 \leq i \leq 4$ are linear. Since $\phi \in \mc^2_b(\R^d)$, and $\mu_t^N$ and $\mu_t$ are compactly supported in velocity, we easily obtain that
		\[
		\lt|\sum_{i=1}^4 I_i^N \rt|\leq C\sup_{0 \leq t \leq T}W_2(\mu^N_t,\mu_t) \to 0 \quad \mbox{as} \quad N \to \infty, \quad \mbox{almost surely}.
		\]
		$\diamond$ Estimate of $I_5^N$: Note that $I_5^N$ is a linear stochastic integral term. Let us denote by
		\[
		H_s^N:=\int_{\R^d \times \R^d} v\cdot\nabla_v \phi(x,v) \lt(\mu_s^N(dx,dv)-\mu_s(dx,dv)\rt).
		\]
		Since $(x,v)\mapsto v\cdot\nabla_v \phi(x,v)$ is locally bounded and Lipschitz, and $\mu_s^N$ converges to $\mu_s$ almost surely(uniformly in time), $H_s^N$ converges to $0$ almost surely (uniformly in time) as $N$ goes to infinity. Note also that $\lt|H_s^N\rt|\leq C\|\nabla_v \phi \|_{L^\infty}$ almost surely. Thus, by the stochastic dominated convergence theorem \cite[Theorem 2.12, Chapter IV]{RY}, we have
		\[
		\lt| \int_0^t H_s^N dB_s \rt| \rightarrow 0 \quad \mbox{as} \quad N\rightarrow \infty \quad \mbox{in probability.}
		\]
		$\diamond$ Estimate of $I_6^N$: We now estimate the nonlinear term. We rewrite 
		$$\begin{aligned}
		I_6^N& =\lt|\int_0^t \int_{\R^d \times \R^d} \nabla_v\phi(x,v)\cdot\lt(F[\mu_s^N](x,v)\mu^N_s(dx,dv)-F[\mu_s](x,v)\mu_s(dx,dv)\rt) ds\rt|\\
		&\leq \lt|\int_0^t \int_{\R^d \times \R^d} \nabla_v\phi(x,v) \cdot \lt(F[\mu_s^N](x,v)-F[\mu_s](x,v)\rt)\mu_s^N(dx,dv)\,ds\rt|\\
		&\quad+ \lt|\int_0^t \int_{\R^d \times \R^d} \nabla_v\phi(x,v) \cdot F[\mu_s](x,v)\lt(\mu^N_s(dx,dv)-\mu_s(dx,dv)\rt)\,ds\rt|,
		\end{aligned}$$
		where the second term on the right hand side goes to $0$ as $N$ goes to infinity almost surely since $\nabla_v\phi \cdot F[\mu_t]$ is locally bounded and Lipschitz, and $\mu_t^N$ converges to $\mu_t$ as $N$ goes to infinity uniformly in time. Indeed, we get
		\[
		\lt|\int_0^t \int_{\R^d \times \R^d} \nabla_v\phi(x,v)\cdot F[\mu_s](x,v)\lt(\mu^N_s(dx,dv)-\mu_s(dx,dv)\rt)ds\rt| \leq C\sup_{0 \leq t \leq T}W_2(\mu^N_t,\mu_t) \to 0,
		\]
		as $N \to \infty$, almost surely. For the estimate of the first term, we rewrite it that as
		$$\begin{aligned}
		&\lt|\int_0^t \int_{\R^d \times \R^d} \nabla_v\phi(x,v) \cdot \lt(F[\mu_s^N](x,v)-F[\mu_s](x,v)\rt)\mu_s^N(dx,dv)\,ds\rt|\\
		&\quad= \bigg|\int_0^t \int_{\R^d \times \R^d} \lt(\int_{\R^d \times \R^d}\nabla_v\phi(x,v) \cdot \psi(|x-y|)(w-v)\mu_s^N(dx,dv)\rt) \lt(\mu_s^N(dy,dw)-\mu_s(dy,dw)\rt)ds\bigg|.
		\end{aligned}$$
		In order to treat this last term we use the following lemma. 
		\begin{lemma}\label{lem_conv2}  Define the functional $G(\mu_t^N)$ on $\R^d \times \R^d$ as
			\[
			G(\mu^N_t):(y,w)\mapsto \lt(\int_{\R^d \times \R^d}\nabla_v\phi(x,v) \cdot \psi(|x-y|)(w-v)\,\mu_t^N(dx,dv)\rt).
			\] 
			Then $\mathbb{P}$-almost surely, for any $t\in [0,T]$, the vector field $G(\mu_t^N)$ is bounded and locally Lipschitz in velocity.
		\end{lemma}
		\begin{proof} For any $\mu\in \mathcal{P}_2(\R^d \times \R^d)$ we first easily find that 
			\[
			\lt|G(\mu)\rt|_{L^\infty} \leq C\|\nabla_v \phi\|_{L^\infty}(1 + |w|)\psi_M,
			\]
			for some $C > 0$. Next for $(y,w), (y',w') \in \R^{2d} \times \R^{2d}$, we obtain 
			$$\begin{aligned}
			|G(\mu)(y,w) - G(\mu)(y',w')| 
			&= \lt|\int_{\R^d \times \R^d} \nabla_v \phi(x,v)\cdot \lt(\psi(x-y)(w-v) - \psi(x-y')(w'-v)\rt)\mu(dx,dv)\rt|\cr
			&\leq \lt|\int_{\R^d \times \R^d} \psi(x-y) \lt(\nabla_v \phi(x,v)\cdot \lt(w - w'\rt)\rt) \mu(dx,dv)\rt|\cr
			&\quad + \lt|\int_{\R^d \times \R^d} \lt(\psi(x-y) - \psi(x-y')\rt) \lt(\nabla_v \phi(x,v)\cdot (w'-v)\rt) \mu(dx,dv)\rt|\cr
			&=: I + J,\cr
			\end{aligned}$$
			where $I$ is easily estimated by.
			\[
			I \leq \psi_M\|\nabla_v \phi\|_{L^\infty}|w - w'|.
			\]
			For the estimate of $J$, we also easily get
			\[
			J\leq \|\psi\|_{Lip}\|\nabla_v \phi\|_{L^\infty} \lt(|w'|+\lt(\int |v|^2\mu(dx,dv)\rt)^{1/2} \rt)|y-y'|.
			\]
			Hence it holds
			$$\begin{aligned}
			|G(\mu_t^N)(y,w) - G(\mu_t^N)(y',w')|\leq C_{\psi,\phi}\lt(1+|w'| + \lt(\int |v|^2\mu_t^N(dx,dv)\rt)^{1/2} \rt) \lt|(y,w) - (y',w')\rt|.
			\end{aligned}$$
			This together with Lemma \ref{lem:EstVel} completes the proof.
		\end{proof}
		Using Lemma \ref{lem_conv2}, since $\mu_t^N$ and $\mu_t$ are compactly supported in velocity and $\mu_t^N$ converges to $\mu_t$ almost surely uniformly in time, we have
		\[
		\lt|\int_0^t \int_{\R^d \times \R^d} \nabla_v\phi(x,v)\cdot \lt(F[\mu_s^N](x,v)\mu^N_s(dx,dv)-F[\mu_s](x,v)\mu_s(dx,dv)\rt)ds\rt|\rightarrow 0 \quad \mbox{as} \quad N \to \infty,
		\]
		almost surely. 
		
		Putting all those estimates together deduces that $I_{\phi,t}(\mu)$ defined in \eqref{eq:I} is equal to sequences which go to $0$ as $N \to \infty$ in probability, thus that is equal to $0$ almost surely. This concludes that $(\mu_t)_{t\in [0,T]}$ satisfies the weak formulation in Definition \ref{def_weak} $(ii)$. We now show that $\mu_t$ is weakly continuous. Note that we can always assume that the process $(\langle \mu_t,\phi\rangle)_{t\in[0,T]}$ is adapted by changing the notion of the filtration $(\mathcal{F}_t)_{t\in [0,T]}$. Indeed, that can be obtained as the limit of the sequence of the adapted processes $(\left\langle \mu^N_t,\phi\right\rangle)_{t\in[0,T]}$ in law. Thus it only remains to establish the existence of a continuous version of the process $(\left\langle \mu_t,\phi\right\rangle)_{t\in[0,T]}$ to complete the proof. Recall that $\mu_t$ obtained above satisfies
		$$\begin{aligned}
		\lal\mu_t,\phi \ral &= \lal \mu_0, \phi\ral + \int_0^t \lal \mu_s, v \cdot\nabla_x \phi + (F[\mu_s]+ \sigma v)\cdot\nabla_v \phi \ral \,ds + \sigma \int_0^t \lal \mu_s, v \otimes  v : \nabla^2_v \phi \ral\,ds  - \sqrt{2\sigma}\int_0^t \lal \mu_s,  v\cdot \nabla_v \phi \ral \,dB_s \cr
		&=:  \lal \mu_0, \phi\ral+\int_0^t \mathcal{L}_s\, ds+ \sqrt{2\sigma}\int_0^t\mathcal{I}_s\,dB_s.
		\end{aligned}$$
		In the rest of this section, we estimate the Lebesgue and it\^o integrals as follows.

		$\diamond$ {\it Estimate of the Lebesgue integral}: Since $\phi \in \mc^2_c(\R^d \times \R^d)$ is compactly supported in both position and velocity, we get
		\[
		\lal \nu, v \cdot\nabla_x \phi+\sigma v\cdot \nabla_v \phi +v \otimes  v : \nabla^2_v \phi \ral \leq C_{\sigma,\phi},
		\] 
		for any $\nu \in \mathcal{P}(\R^d \times \R^d)$, where $C_{\sigma,\phi}$ is a constant which depends only on $\sigma,\phi$. This gives
		\[
		\lt|\mathcal{L}_s\rt|\leq C_{\sigma,\phi}+\lt| \lal \mu_s,F[\mu_s]\cdot\nabla_v \phi \ral \rt|.
		\]
		On the other hand, it follows from Remark \ref{rem:comsup} that $\mu_s$ is compactly supported in velocity. Thus we obtain that the alignment force field is almost surely bounded from above by 
		\[
		\lt| F[\mu_s] \rt| \leq 2\psi_M R^2e^{\psi_M s-2\sqrt{2\sigma}B_s}\lt(R^2\frac{1-e^{-\psi_Ms}}{\psi_M}+1\rt). 
		\]
		Hence we have that $(\int_0^t \mathcal{L}_s\, ds)_{t\in[0,T]}$ is almost surely Lipschitz in time.
		
		$\diamond$ {\it Estimate of the It\^o integral:} Applying Burkholder-Davis-Gundy's inequality yields 
		$$\begin{aligned}
		\E\lt[ \sup_{s<u<t} \lt|\int_s^u \mathcal{I}_s \,dB_s\rt|^{2p} \rt] &\leq C_p \E\lt[ \lt|\int_s^t \mathcal{I}^2_s \,ds\rt|^{p} \rt]= C_p \E\lt[ \lt|\int_s^t \lal \mu_s,  v\cdot \nabla_v \phi \ral ^2 \,ds\rt|^{p} \rt]\cr
		&\leq C_{p,\phi} \E\lt[ \lt|\int_s^t \lal \mu_s,  |v|^2 \ral \,ds\rt|^{p} \rt]\\
		&\leq C_{p,\phi} \E\lt[ \lt|\int_s^t \lal \mu_0,  |v|^2 \ral e^{-2\sqrt{2\sigma}B_t} \,ds\rt|^{p} \rt]\\
		&\leq C_{p,\phi,\mu_0} |t-s|^p \E\lt[  e^{2p\sqrt{2\sigma}\sup_{r\in[0,T]} |B_r|}  \rt],
		\end{aligned}$$
		for any $p>1$ and $s,t \in [0,T]$, where we used Jensen's inequality with $\mu_s\in \mathcal{P}(\R^{2d})$ in the second line, the large-time behavior estimate in Lemma \ref{lem:EstVel} in the third line, and the fact that $\mu_0$ is compactly supported in velocity in the fourth line. On the other hand, classical properties of one dimensional Brownian motion provide 
		\[
		\E\lt[  e^{2p\sqrt{2\sigma}\sup_{r\in[0,T]} |B_r|}  \rt] \leq 2\E\lt[  e^{2p\sqrt{2\sigma}|B_T|}  \rt]\leq C_{\sigma,T,p},
		\]
		putting all those estimates together we find that there is a constant $C_{p,\mu_0,\phi,\sigma,T}$ such that 
		\[
		\E\lt[ \sup_{s<u<t} \lt|\int_s^u \mathcal{I}_s \,dB_s\rt|^{2p} \rt]\leq C_{p,\mu_0,\phi,\sigma,T}|t-s|^p.
		\]
		Finally, we use Kolmogorov's continuity theorem to conclude that there exists a continuous version of the process $(\int_0^t \mathcal{I}_s \,dB_s)_{t\in[0,T]}$. 
		%
		%
		%
		%
		\section{Phase change phenomenon: Proof of Theorem \ref{thm_spde_f}}\label{sec_lt}
		In this section, we provide details of the proof of Theorem \ref{thm_spde_f} on the flocking and non flocking estimates for the stochastic kinetic equation \eqref{main_spde}. As mentioned in Introduction, we employ a similar strategy used for the particle system \eqref{main_sde} proposed in \cite{AH, TLY}. 
		
		Recall the variance functions of stochastic particle velocity fluctuation around $\bar v_t$:
		\[
		E_t = \int_{\R^d \times \R^d} |\bar v_t - v|^2 \mu_t(dx,dv).
		\]
		Note that
		$$\begin{aligned}
		\bar v_t &= \bar v_0 + \int_0^t \int_{\R^d \times \R^d} v\,\mu_s (dx,dv)\,ds \cr
		&=\int_0^t \int_{\R^d \times \R^d} F[\mu_s](x,v)\mu_s(dx, dv)\, ds + \sqrt{2\sigma}\int_0^t \int_{\R^d \times \R^d} (\bar v_s - v)\mu_s(dx, dv) \circ dB_s\cr
		&= 0,
		\end{aligned}$$
		where we used
		$$\begin{aligned}
		\int_{\R^d \times \R^d} F[\mu_s](x,v)\mu_s(dx, dv) &= \int_{\R^{2d} \times \R^{2d}} \psi(|x-y|)(w-v)\mu_s(dx, dv)\mu_s(dy, dw)= 0.
		\end{aligned}$$
		This, together with a straightforward computation, yields
		$$\begin{aligned}
		\frac12 dE_t &= \int_{\R^d \times \R^d} (v - \bar v_t) \cdot F[\mu_t](x,v) \mu_t(dx, dv)\,dt  + 2\sigma \int_{\R^d \times \R^d} |v - \bar v_t|^2 \mu_t (dx, dv)\,dt\cr
		&\quad - \sqrt{2\sigma}\int_{\R^d \times \R^d} |v - \bar v_t|^2 \mu_t(dx, dv)\,dB_t\cr
		&= -\frac12\int_{\R^{2d} \times \R^{2d}}\psi(|x-y|)|w-v|^2 \mu_t(dx, dv)\mu_t(dy, dw)\,dt + 2\sigma E_t\,dt   - \sqrt{2\sigma}E_t \,dB_t.
		\end{aligned}$$
		Thus the process $E_t$ satisfies 
		\begin{align}\label{witho_E}
			\begin{aligned}
				dE_t &= -\int_{\R^{2d} \times \R^{2d}}\psi(|x-y|)|w-v|^2 \mu_t(dx, dv)\mu_t(dy, dw)\,dt + 4\sigma E_t\,dt  - 2\sqrt{2\sigma}E_t \,dB_t.
			\end{aligned}
		\end{align}
		Taking the expectation to the above gives
		\bq\label{gro_E}
		\frac12\frac{d}{dt}\E[E_t] + \frac12\E\lt[\int_{\R^{2d} \times \R^{2d}}\psi(|x-y|)|w-v|^2 \mu_t(dx, dv)\mu_t(dy, dw) \rt] = 2\sigma\E[E_t].
		\eq
		Note that
		\[
		\int_{\R^{2d} \times \R^{2d}}|w-v|^2 \mu_t(dx, dv)\mu_t(dy, dw) = 2 E_t,
		\]
		due to $\bar v_t = 0$. Thus we get
		\[
		\psi_m\E[E_t] \leq \frac12\E\lt[\int_{\R^{2d} \times \R^{2d}}\psi(|x-y|)|w-v|^2 \mu_t(dx, dv)\mu_t(dy, dw) \rt] \leq \psi_M \E[E_t],
		\]
		and this together with \eqref{gro_E} provides
		\[
		-2(\psi_M - 2\sigma)\E[E_t] \leq \frac{d}{dt}\E[E_t] \leq -2(\psi_m - 2\sigma)\E[E_t].
		\]
		Applying Gronwall's inequality, we have
		\[
		\E[E_0]e^{-2(\psi_M - 2\sigma) t} \leq \E [E_t] \leq \E[E_0]e^{-2(\psi_m - 2\sigma) t} \quad t \geq 0.
		\]
		This completes the proof of Theorem \ref{thm_spde_f}.

		As mentioned in Remark \ref{rmk_without}, we can also obtain the convergence of the variance functional $E_t$ without taking the expectaiton even though it does not provide the phase change phenomenon. Let us go back to equation \eqref{witho_E}. Using the similar strategy as before, we estimate the drift term in \eqref{witho_E} as 
		\[
		- 2(\psi_M-2\sigma)E_t dt - 2\sqrt{2\sigma}E_t \,dB_t \leq dE_t \leq - 2(\psi_m -2\sigma)E_t dt - 2\sqrt{2\sigma}E_t \,dB_t, \quad \mbox{a.s.,}
		\]
		Applying It\^o's formula gives
		\[
		d \log E_t = \frac{d E_t}{E_t} - \frac{|dE_t|^2}{2|E_t|^2} = \frac{d E_t}{E_t} - 4\sigma dt \leq -2\psi_m dt - 2\sqrt{2\sigma} dB_t \quad \mbox{a.s.}
		\]
		Taking the time integration to the above inequality gives
		\[
		\log E_t \leq \log E_0 - 2\psi_m t - 2\sqrt{2\sigma}\int_0^t dB_s = \log E_0 - 2\psi_m t - 2\sqrt{2\sigma}B_t,
		\]
		i.e.,
		\bq\label{eqn_w}
		E_t \leq E_0 \exp\lt( -2\psi_m t - 2\sqrt{2\sigma} B_t\rt).
		\eq
		On the other hand, by using the fact that for any $\epsilon > 0$, almost surely, there exists $t_0 > 0$ such that $|B_t| \leq \epsilon t$ for all $t \geq t_0$, we can further estimate 
		\[
		E_t \leq E_0 \exp\lt( -2(\psi_m - \epsilon_0)t \rt) \quad \mbox{a.s.},
		\]
		for some $0<\epsilon_0 < \psi_m$ and $t \geq t_* > 0$.
		
		Summarizing the above discussion, we have the following result. 
		\begin{proposition}\label{prop_without} Let $\mu$ be a measure-valued solution for the equation \eqref{main_spde}. Suppose that there exists a positive constant $\psi_m$ such that $0 < \psi_m \leq \psi(s)$ for $s \in \R_+$. Then we have the following time asymptotic flocking estimate. 
			\[
			\int_{\R^d \times \R^d} |\bar v_0 - v|^2 \mu_t(dx, dv) \to 0 \quad \mbox{as} \quad t \to \infty, \quad a.s.,
			\]
			at least exponentially fast.
		\end{proposition}
		
		\begin{remark} The stochastic process $E_t$ in \eqref{witho_E} resembles a geometric Brownian motion. Note that the inequality \eqref{eqn_w} can be rewritten as
			\[
			E_t \leq E_0 e^{-2\sqrt{2\sigma}B_t^{(\psi_m)/(\sqrt{2\sigma})}},
			\]
			where $B^{\nu}_t$ denotes the one dimensional Brownian motion with constant drift $\nu \in \R$, i.e., $B^\nu_t = B_t + \nu t$.
		\end{remark}
		
		\begin{remark} When there is no multiplicative noise $\sigma=0$, then the system \eqref{main_sde} becomes the original Cucker-Smale model. For the Cucker-Smale model, the unconditional flocking estimate can be obtained if the communication weight $\psi$ is not integrable, see \cite{CHL}. If not, suitable assumptions for the initial configurations are needed for the flocking estimate. In \cite{AH}, the multiplicative noise is considered for the stochastic particle system \eqref{main_sde} in the It\^o sense, and in that case, we can get the flocking behavior for any nonnegative communication weight $\psi$, i.e., taking into account the multiplicative noises in the Cucker-Smale model enable us to have the unconditional flocking estimate. This implies that the multiplicative noise in the It\^o sense plays a role as a (stochastic) control which enhances the flocking behavior of particles. 
		\end{remark}
		
		%
		%
		%
		%

		\section*{Acknowledgments}
		Y-PC is supported by NRF 2017R1C1B2012918 and 2017R1A4A1014735. 
		
		%
		%
		%
		%

	\end{document}